\documentclass[12pt]{amsart}

\usepackage[utf8]{inputenc}

\usepackage{eucal}
\usepackage{upgreek}
\usepackage{pdfsync}
\usepackage[all, cmtip]{xy}
\usepackage{amsfonts}
\usepackage{amsmath}
\usepackage{amssymb}
\usepackage{amscd}
\usepackage{color,xcolor}
\usepackage{amsthm}
\usepackage{amsxtra}
\usepackage{bm}
\usepackage{bbm}
\usepackage{graphicx}
\usepackage{mathrsfs}
\usepackage[numbers,sort]{natbib}
\usepackage{isomath}
\usepackage{enumitem}
\usepackage{latexsym} 
\usepackage{dsfont}
\usepackage{url}
\usepackage{mathtools}
\usepackage{soul}

\newtheorem{thm}{Theorem}

\newtheorem{prop}{Proposition}

\newtheorem{cor}[prop]{Corollary}

\theoremstyle{definition}

\theoremstyle{remark}

\newtheorem{rmk}[prop]{Remark} 
\newtheorem{exa}[prop]{Example}

\def\mrm#1{{\mathrm{#1}}}
\def\bb#1{{\mathbb{#1}}}
\def\cl#1{{\mathcal{#1}}}

\newcommand{\R}{{\mathbb{R}}}
\newcommand{\Z}{{\mathbb{Z}}}
\newcommand{\C}{{\mathbb{C}}}
\newcommand{\Q}{{\mathbb{Q}}}
\newcommand{\N}{{\mathbb{N}}}

\newcommand{\bK}{{\mathbb{K}}}
\newcommand{\HH}{{\mathbb{H}}}

\newcommand{\rC}{{\mathrm{C}}}

\newcommand{\al}{\alpha}

\newcommand{\La}{\Lambda}

\newcommand{\Om}{\Omega}

\newcommand{\cA}{\mathcal{A}}

\renewcommand{\bar}[1]{\overline{#1}}

\DeclareMathOperator{\ind}{\mathrm{ind}}

\renewcommand{\null}{\mathrm{null}}

\DeclareMathOperator{\coker}{\mathrm{coker}}

\def\H2{H^{(2)}}
\def\F{\mathbb F}

\usepackage[hyperindex]{hyperref}

\newcommand{\esemail}{egor.shelukhin@umontreal.ca}

\begin{document}
	
\title[On non-contractible closed geodesics and homotopy groups]{Remark on non-contractible closed geodesics \\ and homotopy groups}

\author{Egor Shelukhin}
\email{\esemail}
\address{Department of Mathematics and Statistics, University of Montreal, C.P. 6128 Succ.  Centre-Ville Montreal, QC H3C 3J7, Canada}

\author{Jun Zhang}
\email{jzhang4518@ustc.edu.cn}
\address{The Institute of Geometry and Physics, University of Science and Technology of China, 96 Jinzhai Road, Hefei Anhui, 230026, China}

\bibliographystyle{abbrv}

\begin{abstract}
We prove that if the $m$-th homotopy group for $m \geq 2$ of a closed manifold has non-trivial invariants or coinvariants under the action of the fundamental group, then there exist infinitely many geometrically distinct closed geodesics for a $C^4$-generic Riemannian metric. If moreover there are infinitely many conjugacy classes in the fundamental group, then the same holds for every bumpy Riemannian metric. Finally, under a homological condition that can detect the non-triviality of the higher homotopy groups, the same conclusion holds for every Riemannian metric. All three conclusions extend to smooth, possibly non-reversible Finsler metrics. In the generic statement, the conclusion holds for $C^6$-topology; under the homological hypothesis, the conclusion holds for every smooth Finsler metric.
\end{abstract}


\maketitle

\section{Introduction and main results}

\subsection{Introduction}

The question of the existence of infinitely many geometrically distinct closed (periodic) geodesics on Riemannian or Finsler manifolds was studied extensively in the past. This question seeks to determine whether or not for every metric or for a class of metrics on a given manifold $M$ there exist infinitely many {\em prime} closed geodesics: those closed geodesics that are not obtained as iterations of others. 

The efforts were primarily focused on the simply-connected case, where by \cite{GromollMeyer, VPSullivan, JM16} a manifold with finitely many prime closed geodesics must have $\bK$-cohomology, for $\bK = \Q$ or $\bK = \F_p,$ generated as a unital $\bK$-algebra by a single element. Manifolds satisfying this condition for some $\bK$ do exist: notable examples are the compact rank-one symmetric spaces (CROSS): $S^n, \C P^n, \R P^n, \HH P^n, \bb OP^2$. For such spaces it is known \cite{Rad94} that a $C^2$-generic metric must have infinitely many prime closed geodesics. Prior work \cite{KT72,Hi84} proves the same for $C^4$-generic metrics: \cite{KT72} implies that $C^4$-generically on any closed manifold there are infinitely many prime closed geodesics, unless they are all hyperbolic, while \cite{Hi84} proves that for manifolds rationally homotopy equivalent to a CROSS other than $\R P^n,$ given that all prime closed geodesics are hyperbolic, there are infinitely many of them. The case of $\R P^n$ follows by a quick covering argument, as does the case of any closed manifold with finite fundamental group, or by the paper \cite{BTZ81} discussed below. A celebrated result \cite{Bangert93, Franks92, Hingston93} proves that every Riemannian metric on $S^2$ has infinitely many prime closed geodesics, and the same is true for $\R P^2$. The analogous statement for any other CROSS is a well-known open question. Recently, AI-generated claimed solutions to this question have appeared online.

The case of infinite fundamental groups at first appears to be simpler, as by a classical theorem often attributed to Cartan or to Hilbert (see for example \cite[Chapter 12, Theorem 2.2]{DC92-book}), in every non-trivial free homotopy class of loops on $M$ there exists a closed geodesic. However, firstly, it is an open question in group theory \cite{Makowsky, BMS01} whether an infinite finitely presented group could have finitely many conjugacy classes. We recall that the conjugacy classes in $\pi_1(M)$ correspond to free homotopy classes of loops on $M$. Secondly, even if we assume that there are infinitely many conjugacy classes, the simplification appears to be quite illusory, starting with the fact that non-homotopic closed geodesics might not be geometrically distinct (see also \cite{Taim10}). Existence of infinitely many prime closed geodesics is currently known under stronger conditions involving the fundamental group: for example \cite{BH84} if the manifold is not the circle and its fundamental group is $\Z$ (see \cite{Tai85,Tai93, RT22} for other conditions of this kind and \cite{Gro2000, Bal86} for prior work). Generic existence (in $C^4$ topology) of infinitely many prime closed geodesics was proved under special conditions on $\pi_1(M)$ in \cite{BTZ81}: for instance this covers the case where $\pi_1(M)$ has finitely many conjugacy classes, yet is not simply-connected. Furthermore, \cite{RT22} proves $C^4$-generic existence of infinitely many prime closed geodesics on every closed $3$-manifold, and produces further results about closed geodesics on connected sums of closed manifolds (see also \cite{PP05,La01}). Finally, the recent work \cite{CM25} proves the generic existence of infinitely many closed geodesics on manifolds with non-zero rational first Betti number.

Our main result, Corollary \ref{cor: generic}, shows that $C^4$-generic existence of infinitely many prime closed geodesics holds on closed manifolds of any dimension, under a topological condition regarding the action of $\pi_1(M)$ on $\pi_m(M)$ for $m \geq 2,$ in the spirit of  Bangert-Hingston \cite{BH84}, Albers-Frauenfelder-Oancea \cite{AFO17} and Taimanov \cite{Tai85}, which have inspired this paper. Namely, we assume that this action for some $m \geq 2$ has either non-trivial invariants or non-trivial coinvariants. For example, this applies already to manifolds of the form $M= X \times S^m$ for $m \geq 2,$ for many of which our result is new.

In Theorem \ref{thm: inv coinv} below, we also prove that if, in addition, there are infinitely many free homotopy classes of loops on $M,$ then there exist infinitely many prime closed geodesics for every {bumpy} Riemannian metric on $M.$  {We recall that a Riemannian metric is called {\em bumpy} if every closed geodesic is non-degenerate, that is, the orbit of every closed geodesic under the natural circle action is a non-degenerate critical submanifold of the energy functional in the sense of Morse--Bott; bumpy metrics form a residual subset in the $C^4$ topology by the bumpy metric theorem \cite{Abraham, Anosov-bumpy}. Finally, under an additional condition, detecting non-trivial homotopy groups by means homology, we prove the same statement for {\em every} Riemannian metric on $M$ (see Theorem \ref{thm: every metric} below).

\subsection{Main results}

\begin{thm}\label{thm: inv coinv}
Let $M$ be a closed manifold. Suppose that there exists $x \in \pi_m(M),$ $x \neq 0,$ for $m\geq 2$ such that $a \cdot x = x$ for all $a \in \pi_1(M)$ or a homomorphism $\xi: \pi_m(M) \to A,$ $\xi \neq 0$ to an abelian group $A$ such that $a^* \xi = \xi$ for all $a \in \pi_1(M).$ Moreover, let the set $\pi_1(M)/\mrm{conj}$ of conjugacy classes in $\pi_1(M)$ be infinite.
Then for every {bumpy} Riemannian metric on $M$ there exist infinitely many geometrically distinct closed geodesics.\end{thm}

{
\begin{rmk}\label{rmk: why bumpy}
The bumpiness assumption in Theorem \ref{thm: inv coinv} will be used at a single point of the proof: our argument in Section \ref{sec-proof} produces a closed geodesic $\eta$ which is not a global minimum of the energy in its connected component and whose iterates all have Morse index zero, and then we conclude by the theorem of Bangert--Klingenberg \cite[Theorem 3]{BK83}. However, that theorem requires in addition that $\eta$ to be {\em homologically visible}, that is, its integral local homology defined as follows, 
\begin{equation} \label{dfn-int-loc-hom}
\rC_*(\eta;\Z) :=\; H_*\big( \{\cA < \cA(\eta)\} \cup S(\eta),\, \{\cA < \cA(\eta)\};\Z \big) 
\end{equation}
to be non-zero. For a bumpy metric, this homological visibility is automatic. Indeed, the critical orbit $S(\eta)$ is non-degenerate by definition and with index zero (by our argument), so $\rC_*(\eta;\Z) \cong H_*(S(\eta); \Z) = H_*(S^1;\Z) \neq 0$. Note that for a degenerate metric, on a contrary, zero index and non-minimality in general do {\em not} imply homological visibility (for instance, the toy model $f(x)=x^3$ at the origin, whose local critical group vanishes). In fact, (over non-compact complete surfaces) Allais--Soethe \cite{AllaisSoethe} construct examples where a index zero (and non-minimal) closed geodesic could be homologically invisible. \end{rmk}
}

\begin{rmk}\label{rmk: fund gp loc sys}
We remark that our condition on the higher homotopy groups does not constrain the fundamental group of the manifold. For instance, taking a product with the two-sphere does not change the fundamental group but ensures that this condition holds (see also Remark \ref{rmk: constructions}). The condition on coinvariants for $m=2$ and $A = GL(1, \F_p)$ for a finite field $\F_p$ is satisfied if and only if there exists a non-trivial rank-one $\F_p$-local system on $\cl LM$ with a trivial restriction to $M$ via the constant loop embedding $M \to \cl LM.$ We refer to \cite[Section 2.2]{AFO17} for a related discussion and \cite[Proposition 9]{AFO17}, showing that this condition, for a suitable prime $p,$ is in turn implied by the non-triviality of the image $H_2^S(M;\Z)$ of the second Hurewicz map $\pi_2(M) \to H_2(M;\Z).$ 
\end{rmk}

Recall that the natural map $\pi_1(M) \to \pi_0(\cl LM)$ from the fundamental group of $M$ to the set  $\pi_0(\cl LM)$ of connected components of the free loop space $\cl LM$ of $M$ induces a canonical isomorphism of sets, $\pi_1(M)/\mrm{conj} \xrightarrow{\cong} \pi_0(\cl LM),$ where $\pi_1(M)/\mrm{conj}$ is the set of conjugacy classes in $\pi_1(M).$ Hence requiring that the set $\pi_1(M)/\mrm{conj}$ is infinite is the least possible assumption on the richness of the fundamental group that would help establishing the existence of infinitely many geometrically distinct closed geodesics. In fact, a well-known open question in group theory, due to Makowsky \cite{Makowsky} (see also \cite{BMS01}), is whether every infinite finitely-presented group has infinitely many conjugacy classes. Note that the assumption on the conjugacy classes yields infinitely many non-homotopic closed geodesics, as by Cartan and Hilbert's theorem, there is a closed geodesic in every non-trivial free homotopy class $\al$ of loops on $M.$ (It is given as point of global minimum of energy in the corresponding connected component $\cl L_{\al}M$ of $\cl LM$.) However, these geodesics might easily not all be {\em geometrically distinct} since a geodesic $c$ and its $k$-th iterate $c^k,$ given by $c^k(t) = c(kt),$ could contribute to different free homotopy classes. In general, the free homotopy class $[c^k] \in \pi_0(\cl LM)$ of the iterate $c^k$ might depend on $k$ in a complicated way.

\begin{cor}\label{cor: generic}
Let $M$ be a closed manifold. Suppose that there exists $x \in \pi_m(M),$ $x \neq 0,$ for $m\geq 2$ such that $a \cdot x = x$ for all $a \in \pi_1(M)$ or a homomorphism $\xi: \pi_m(M) \to A,$ $\xi \neq 0$ to an abelian group $A$ such that $a^* \xi = \xi$ for all $a \in \pi_1(M).$ Then for a $C^4$-generic Riemannian metric on $M$ there exist infinitely many geometrically distinct closed geodesics.
\end{cor}

We remark that in this corollary we make no assumption on conjugacy classes in the fundamental group. It appears to be new even for manifolds of the form $M=X \times S^2,$ where $X$ is {\it an arbitrary} closed manifold (see also Remark \ref{rmk: constructions}). Of course $X$ for which prior methods do not apply are more restricted; for instance there should be an element $\al \in \pi_1(X)/\mrm{conj}$ such that $\al^k,$ $k \geq 1,$ are all distinct. 

\begin{proof}[Proof of Corollary \ref{cor: generic}]
If $M$ is simply connected, the conclusion is well-known by the work of Klingenberg-Takens \cite{KT72}, Hingston \cite{Hi84}, and Rademacher \cite{Rad94}. If $M$ is not simply connected, consider a non-trivial class $a \in \pi_1(M).$ Then either there exist two equal conjugacy classes $[a^k]=[a^l]$ for some positive integers $0<k<l,$ in which case the conclusion follows by work of Ballmann-Thorbergsson-Ziller \cite[Theorem A]{BTZ81}, or the conjugacy classes $[a^k],$ for $k$ a positive integer, are all distinct, in which case {$\pi_1(M)/\mrm{conj}$ is infinite and the conclusion follows by Theorem \ref{thm: inv coinv} applied to every bumpy metric, since bumpy metrics form a residual subset of the space of Riemannian metrics in the $C^r$ topology for all $2 \leq r \leq \infty,$ hence in $C^4$ topology \cite{Abraham,Anosov-bumpy} in particular.} 
\end{proof}

\begin{rmk}
By the arguments proving \cite[Corollary B]{Contreras}, one can upgrade Corollary \ref{cor: generic} to the stronger conclusion that for $C^4$-generic metrics, the number of prime periodic orbits of length at most $T$ grows exponentially in $T.$ (Note that while this is not stated explicitly, \cite[Corollary B]{Contreras} is proven therein for simply-connected manifolds.)
\end{rmk}

\begin{rmk}
It is not hard to see that Theorem~\ref{thm: inv coinv} (as well as Theorem~\ref{thm: every metric} and Corollary~\ref{cor: hurewicz} below) extends to the setting of smooth, possibly non-reversible, Finsler metrics by a finite-dimensional approximation approach (see \cite{Rad92-book} as well as \cite{RT22, RT22b}). The key input is a Finsler version of \cite[Theorem~3]{BK83}, which can be obtained by the argument in \cite[\S7.2]{Rad92-book}, especially its Satz~7.4. Moreover, Corollary~\ref{cor: generic} also extends to smooth Finsler metrics with the $C^6$-topology, using \cite[Theorems~3, 4 and~6, and Proposition~4]{RT22b}. In particular, the regularity $C^6$ comes from \cite[Theorem 6(a)]{RT22b}. Finally, recall that closed geodesics of a Finsler metric correspond precisely to closed Reeb orbits of the canonical contact form on its unit cosphere bundle. In forthcoming work \cite{SSZ}, we will establish an analogue of Theorem \ref{thm: inv coinv} for a broader class of contact manifolds. To replace the role of \cite[Theorem~3]{BK83} in this contact setting, we will build up a symplectic framework based on filtered Floer theory and follow a barcode machinery from \cite{She-HZ}.
\end{rmk}

\begin{rmk}\label{rmk: constructions}
It is easy to see that the family $\cl N$ of manifolds satisfying the condition of Corollary \ref{cor: generic} is an ideal in the family $\cl M$ of all closed manifolds with respect to the Cartesian product. In certain situations, this property is also preserved under connected sums. For instance, consider the family $\cl H_2 \subset \cl M$ of all closed manifolds $M$ satisfying $H_2^S(M;\Z) \neq 0$ (see Remark \ref{rmk: fund gp loc sys}). Set $\cl H^{\geq m}_2 = \cl H_2 \cap \cl M^{\geq m},$ where $\cl M^{\geq m} \subset \cl M$ is the subfamily of manifolds of dimension at least $m.$  Let $m=3.$ Then $\cl H^{\geq 3}_2$ is an ideal in $\cl M^{\geq 3}$ with respect to connect sum: if $X \in \cl M^{\geq 3}$ and $Y \in \cl H^{\geq 3}_2$ then $X \# Y \in \cl H^{\geq 3}_2.$ This is a consequence of the fact that for manifolds $X,Y \in \cl M^{\geq 3},$ by the homotopy exact sequence of a pair, $\pi_2(X \# Y) \cong \pi_2(X) \oplus \pi_2(Y),$ and the Mayer-Vietoris sequence for the natural covering $X' \cup Y' = X \# Y,$ where $X' \cong X \setminus B_X, Y' \cong Y \setminus B_Y$ for closed balls $B_X \subset X, B_Y \subset Y.$ In all dimensions $n$ other than $n \in \{0, 1, 3\}$ there exist simply-connected manifolds $X^n \in \cl H_2$ of dimension $n.$ Indeed, for $n=2,$ $X^2=S^2$ and for $n\geq 4,$ we can take $X^n = S^2 \times S^{n-2}.$ Combining with the previous point, we see by the Van Kampen theorem, that for every manifold $M \in \cl M^{\geq 4}$ there exists a manifold $M' \in \cl H^{\geq 4}_2$ with $\pi_1(M') \cong \pi_1(M).$ We can take $M' = M \# X^n.$ Similarly, for all $M \in \cl M,$ $M'' = M \times X^n \in \cl H_2$ satisfies $\pi_1(M'') \cong \pi_1(M).$
\end{rmk}

\begin{rmk}
The conclusion of Corollary \ref{cor: generic} also holds for smooth Riemannian metrics equipped with the $C^4$-topology. Here, since $\{\mbox{$C^\infty$-metrics}\} \subset \{\mbox{$C^4$-metrics}\}$, the $C^4$-topology on the space of smooth Riemannian metrics is the {\it relative} topology, i.e., induced by the $C^4$-topology on the space of $C^4$-Riemannian metrics. More precisely, for brevity, let $\mathcal R^r(M)$ denote the space of $C^r$-Riemannian metrics on $M$ (where $r$ may equal $\infty$). Then, under the same hypotheses as in Corollary~3, we claim that the set of all smooth Riemannian metrics that admit infinitely many geometrically distinct closed geodesics contains a dense $G_\delta$ subset of $\mathcal R^\infty(M)$, with respect to the relative $C^4$-topology\footnote{Recall that a $G_\delta$ set is a countable intersection of open subsets (note that a $G_{\delta}$ set is not necessarily open or closed). Moreover, if the set $\bigcap_{i \in \N} U_i$ is dense, then each $U_i$ appearing in this intersection is also dense. Furthermore, with respect to this relative $C^4$-topology, a subset $U \subset \mathcal R^{\infty}(M)$ is open precisely when there exists an open subset $V \subset \mathcal R^4(M)$ such that $U = V \cap \mathcal R^{\infty}(M)$.}. In fact, we prove that this set contains the intersection of countably many $C^4$-open and $C^\infty$ dense subsets of $\mathcal R^{\infty}(M).$ In particular, this dense $G_{\delta}$ subset is a residual subset of $\mathcal R^{\infty}(M)$ equipped with the relative $C^4$-topology. In particular, the conclusion of Corollary \ref{cor: generic} holds in a $C^4$-generic sense among smooth Riemannian metrics. 

We give a brief justification of this claim without assuming that $\mathcal R^\infty(M)$ is a Baire space in the relative $C^4$-topology\footnote{Recall that both $\mathcal R^\infty(M)$ with the $C^{\infty}$-topology and $\mathcal R^4(M)$ with the $C^{4}$-topology are complete metrizable spaces, so they are Baire spaces. Moreover, $\mathcal R^4(M)$ with the $C^{4}$-topology is an open subset of a Banach space, while $\mathcal R^\infty(M)$ with the $C^{\infty}$-topology is an open subset of a Fr\'echet space.}. For each $N\geq 1$, set
\[
    \mathcal U_N(M): = \left\{ g\in\mathcal R^\infty(M) \, \bigg|\, 
        \begin{array}{l}
        \mbox{there are at least $N$ geometrically distinct}\\
        \mbox{non-degenerate closed geodesics on $(M, g)$}
        \end{array} \right\}.
\]
First, we show that $\mathcal U_N(M)$ is open in the relative $C^4$-topology. For any smooth metric $g_0 \in \mathcal U_N(M)$ (viewed as an element of $\mathcal R^4(M)$), the implicit function theorem, applied in $\mathcal R^4(M)$ with the $C^4$-topology, implies that any non-degenerate closed geodesic of $g_0$ persists as a nearby non-degenerate closed geodesic for every sufficiently $C^4$-close metric $g$. In fact, $C^2$-closeness suffices, since the corresponding local
Poincar\'e return maps depend continuously in the $C^1$-topology on the metrics in the $C^2$-topology. Then, for any such metric $g$ that is $C^4$-close to $g_0$, $(M, g)$ also admits at least $N$ geometrically distinct non-degenerate closed geodesics. In particular, if $g$ is smooth, then $g \in \mathcal U_N(M)$. In this way, we obtain a $C^4$-neighborhood $\mathcal V_N(M)$ of $g_0$ such that 
\[ g_0 \in \mathcal V_N(M) \cap \mathcal R^{\infty}(M) \subset \mathcal U_N(M).\]
Therefore, by definition, $\mathcal U_N(M)$ is open in this relative $C^4$-topology. 

Next, we show that each $\mathcal U_N(M)$ is dense in the relative $C^4$-topology. It suffices to show that $\mathcal U_N(M)$ is dense in the $C^{\infty}$-topology, since the $C^{\infty}$-topology is finer (containing more open subsets) than the relative $C^4$-topology on $\mathcal R^{\infty}(M)$. We divide the argument into two cases, as in the proof of Corollary~3. Let $\mathcal B^{\infty}(M)$ denote the set of smooth bumpy metrics. Under the hypotheses of Theorem~A, we have shown that every metric in this set admits infinitely many geometrically distinct closed geodesics. Therefore, $\mathcal B^{\infty}(M) \subset \mathcal U_N(M)$ for any $N$. With respect to the $C^{\infty}$-topology, $\mathcal B^{\infty}(M)$ is residual, so it is dense (since $\mathcal R^{\infty}(M)$ with the $C^{\infty}$-topology is a Baire space). Then $\mathcal U_N(M)$ is $C^{\infty}$-dense. In the remaining cases of the proof of Corollary~3, either Theorem~1.3 in \cite{Rad94} (for compact simply connected manifolds) or Theorem~A in \cite{BTZ81} (for the case where a nontrivial element $a\in\pi_1(M)$ satisfies
$[a^k]=[a^\ell]$ for some integers $0<k<\ell$) provides a residual subset, denoted by $\mathcal E^{\infty}(M) \subset \mathcal R^{\infty}(M)$, in the $C^{\infty}$-topology. Note that the statements of both theorems explicitly include the case $r = \infty$. For each $g \in \mathcal E^{\infty}(M)$, $(M, g)$ admits infinitely many geometrically distinct closed geodesics. Therefore, 
\[ \mathcal E^{\infty}(M) \cap \mathcal B^{\infty}(M) \subset \mathcal U_N(M) \,\,\,\,\mbox{for any $N$}.\]
 Since both $\mathcal E^{\infty}(M)$ and $\mathcal B^{\infty}(M)$ are residual, their intersection is also residual and hence dense in the $C^{\infty}$-topology. Therefore, $\mathcal U_N(M)$ is $C^{\infty}$-dense. 
 
 Finally, our desired $G_{\delta}$ subset is simply $\bigcap_{N \geq 1} \mathcal U_N(M)$. Although it is a countable intersection of (relative) $C^4$-open and (relative) $C^4$-dense subsets, $\mathcal R^{\infty}(M)$ with the relative $C^4$-topology may not be a Baire space. It is therefore not immediate that this $G_{\delta}$ subset is dense in the relative $C^4$-topology. Instead, we draw the desired conclusion in the following way: each $\mathcal U_N(M)$ is both $C^{\infty}$-open (since relative $C^4$-openness implies $C^{\infty}$-openness) and $C^{\infty}$-dense. Thus, in the Baire space $\mathcal R^{\infty}(M)$ equipped with the $C^{\infty}$-topology, the intersection $\bigcap_{N \in \N} \mathcal U_N(M)$ is $C^{\infty}$-dense. Therefore, it is dense in the coarser relative $C^4$-topology. This completes the argument. 
\end{rmk}

\begin{rmk}
One could imagine a generalization of Corollary \ref{cor: generic} which would hold for all closed non-aspherical manifolds, and a suitable generalization of Theorem \ref{thm: inv coinv}. However, the complicated algebraic properties of the group rings $\Z[\pi_1(M)],$ which are, for instance, often non-Noetherian \cite{Non-noetherian}, seem to be an obstruction to proving such results. 
\end{rmk}

{Our second main result concludes that there are infinitely many closed geodesics for {\it every} Riemannian metric, under an additional homological hypothesis that fills up the missing local-homology as discussed in Remark \ref{rmk: why bumpy}. For a Riemannian metric $g$ and $\al \in \pi_0(\cl LM)$ non-trivial, denote by $\rC_{\al}(g) \subset \La_{\al}M$ the set of absolute minimum of the energy functional $\mathcal A$ on $\La_{\al}M$.}

{
\begin{thm}\label{thm: every metric}
Let $(M,g)$ be a closed Riemannian manifold such that $\pi_1(M)/\mrm{conj}$ is infinite. Suppose that there exist an integer $L \geq 1$ and infinitely many classes $\al \in \pi_0(\cl LM)$ such that for each of them
\begin{equation}\label{eq: rel homology hyp} H_j\big(\La_{\al}M, \rC_{\al}(g);\Z\big) \neq 0 \quad \text{for some } 1 \leq j = j(\al) \leq L. \end{equation}
Then $(M,g)$ has infinitely many geometrically distinct closed geodesics.
\end{thm}
}

{
\begin{cor}\label{cor: hurewicz}
Let $M$ be a closed manifold with $\pi_1(M)/\mrm{conj}$ infinite. Suppose that there exist $m \geq 2$ and $x \in \pi_m(M),$ $x \neq 0,$ with $a \cdot x = x$ for all $a \in \pi_1(M),$ whose image $h(x) \in H_m(M;\Z)$ under the $m$-th Hurewicz map is non-zero. Then for every Riemannian metric on $M$ there exist infinitely many geometrically distinct closed geodesics.
\end{cor}
}

{
\begin{proof}[Proof of Corollary \ref{cor: hurewicz}]
For every $\al = [a] \in \pi_0(\cl LM)$ for some element $a \in \pi_1(M)$, by the exactness in Proposition \ref{prop: les} and the invariance of $x$ (that is, $x \in \ker(1-a)$ for all $a$), there exists $y_{\al} \in \pi_m(\La_{\al}M)$ with $ev_*(y_{\al}) = x$. By the naturality of the Hurewicz map, 
\begin{equation} \label{non-vanishing}
ev_* (h(y_{\al})) = h(x) \neq 0 \in H_m(M;\Z) \,\,\,\,\mbox{which implies $H_m(\La_{\al}M;\Z) \neq 0$} 
\end{equation}
for every $\al.$ Now suppose there are only finitely many prime closed geodesics, say $\beta_1, ..., \beta_N$. For every non-trivial $\al,$ the set $\rC_{\al}(g)$ satisfies 
\begin{equation} \label{finiteness}
 \rC_{\al}(g) \subset \{S(\beta_i^{k}) \,\,\mbox{for some $1 \leq i \leq N$ and $k \in \N$ such that $[\beta_i^k] = \alpha$} \}. 
 \end{equation}
Here, $S(\beta_i^{k})$ denotes the critical $S^1$-orbit of the closed (iterated) geodesic $\beta_i^k$. Without loss of generality, assume $\alpha$ is a non-trivial class, and by the defining property of $\rC_{\al}(g)$, we will also assume that each $\beta_i^k$ in (\ref{finiteness}) is an absolute minimum, in particular, $\mathcal A(\beta_i^k) = \min_{\La_{\al}M} \mathcal A$. Then 
\[ k^2 \mathcal A(\beta_i) = \min_{\La_{\al}M} \mathcal A \,\,\,\,\mbox{and it yields $k = \sqrt{\frac{\min_{\La_{\al}M} \mathcal A}{\mathcal A(\beta_i)}}$}. \]
Therefore, for each $1 \le i \le N$, there is at most one solution, which implies that $\rC_{\al}(g)$ is a {\it finite} disjoint union of critical circles, so $H_m(\rC_{\al}(g);\Z) = 0$ as $m \geq 2$. Then the long exact sequence of the pair $(\La_{\al}M, \rC_{\al}(g))$ yields an injective map: $H_m(\La_{\al}M;\Z) \hookrightarrow H_m(\La_{\al}M, \rC_{\al}(g);\Z)$. In particular, we have $H_m(\La_{\al}M, \rC_{\al}(g);\Z) \neq 0$ by (\ref{non-vanishing}). Hence, the hypothesis \eqref{eq: rel homology hyp} of Theorem \ref{thm: every metric} holds with $L = m$ for all non-trivial $\al$ (and there are infinitely many of such classes $\al$, due to our hypothesis on $\pi_1(M)/\mrm{conj}$), and Theorem \ref{thm: every metric} then yields a desired contradiction.
\end{proof}
}

{
\begin{rmk} Compared with the hypothesis in Theorem \ref{thm: inv coinv} and Corollary \ref{cor: generic}, only the invariant class $x \in \ker(1-a)$ is involved in the hypothesis in Corollary \ref{cor: hurewicz}. The coinvariant class $\xi: \pi_m(M) \to A$ with $a^*\xi = \xi$ (for all $a \in \pi_1(M)$) does not match with the Hurewicz map. Explicitly, the existence of such a non-zero $\xi$ implies that for any class $\al \in \pi_0(\cl LM)$, there exists a non-zero element $z_{\al} \in \pi_{m-1}(\Lambda_{\al} M)$. However, in general it does not guarantee that $h(z_{\alpha}) \neq 0$ in $H_{m-1}(\Lambda_{\al}M; \Z)$ under the $(m-1)$-th Hurewicz map $h: \pi_{m-1}(\Lambda_{\al} M) \to H_{m-1}(\Lambda_{\al}M; \Z)$. Therefore, this won't yields the non-vanishing result as in (\ref{non-vanishing}). 
\end{rmk}
}

{
\begin{exa}\label{exa:new} There are abundant examples that satisfy the hypothesis of Corollary \ref{cor: hurewicz}. Here, we emphasize the following specific one (cf.~the discussion right below Corollary \ref{cor: generic}): $M = X \times S^2$ where $X$ is a hyperbolic rational 3-sphere. Such an $X$ can be obtained from a hyperbolic $(p,q)$ Dehn filling on a hyperbolic knot complement in $\mathbb S^3$ (see \cite{GMM09} for a well-known one called the Fomenko–Matveev–Weeks manifold). In particular, the first integer homology group of any such $X$ is torsion over $\Z$, in particular, finite. Therefore, the resulting rational first Betti number $b_1(M) = 0$. 

Let us confirm that our proposed $M$ above satisfies the hypothesis of Corollary \ref{cor: hurewicz}. First, it is a standard fact that $\pi_1(M) = \pi_1(X)$ is torsion-free (see, for instance, Theorem 8.2.1 in \cite{Ratcliffe_book}), which implies that $\pi_1(M)/{\rm conj}$ contains infinitely many classes. Explicitly, $[a^k] \neq [a^l]$ for any $k \neq l$. This can be proved by using the translation length on a group equipped with the word metric, which depends only on the conjugacy classes. One shows that the lengths of conjugacy classes $[a^k]$ grow without bound, forcing the existence of infinitely many distinct conjugacy classes. For more details, see 3.15 Proposition in Chapter III.{\rm $\Gamma$} in \cite{BH99}. Second, take $m=2$ and class $x = [S^2] \in \pi_2(M) = \pi_2(X) \oplus \Z$. Since $\pi_1(M) = \pi_1(X)$ acts on $\pi_2(M)$ diagonally, on the $\pi_2(S^2)$-component this action is trivial, so $x$ is an invariant. Third, by K\"unneth formula, we know 
\[ \Z\left< \{[{\rm pt}\} \times S^2] \right> \subset H_2(M; \Z). \]
Meanwhile, by definition, $h(x) = [\{{\rm pt}\} \times S^2]$, so in particular $h(x) \neq 0$. Therefore, Corollary \ref{cor: hurewicz} applies to our $M$, and then for every Riemannian metric on $M$ there exist infinitely many geometrically distinct closed geodesics.

We remark that our example seems {\it not} to be covered by any prior results. Indeed, $b_1(M) = 0$ rules out the classical sufficient condition $b_1(M) \geq 2$ as well as the recent work \cite{CM25}. The fact that the fundamental group $\pi_1(M)$ is torsion-free as mentioned above rules out \cite{BTZ81}. Moreover, in our case, $\pi_1(M)$ is non-solvable (so in particular, non-abelian) and thus rules out \cite{Tai93,BH84}. Here, the non-solvability can be directly derived from the Tits alternative in \cite{Tits72}, which when applied to $\pi_1(M) \subset  {\rm PSL}(2, \C)$ implies that a non-abelian free subgroup $\mathbb F_2 \subset \pi_1(M)$ (and $\mathbb F_2$ is non-solvable). Our conclusion that covers every Riemannian metric rules out \cite{Tai85}. Finally, \cite{RT22,RT22b} do not apply in our setting since one is for manifold in the form of connected sum, while the other only conclude the second closed geodesics.
\end{exa}
}

\section{Preliminaries}

\subsection{Average index and homotopy groups.}
Let $(M,g)$ be a Riemannian manifold of dimension $\dim M = n.$ Let $\ind(c)$ denote the index of a closed geodesic $c$ in $(M,g).$ The average index (or mean-index) of $c$ is defined as \[ \Delta(c) = \lim_{m\to \infty} \frac{\ind(c^m)}{m} \in \R.\] It satisfies the following useful properties (see \cite[Proposition 6.1, Equation (6.1.3)]{GH09} and \cite[Korollar 4.4]{Rad92-book} for instance):

\begin{enumerate}
\item\label{homogeneity} $\Delta(c^m) = m \Delta(c)$ for all $m \in \Z_{>0}$
\item\label{distance to ind} $|\Delta(c)-\ind(c)| \leq n-1,$ $|\Delta(c)-(\ind(c)+\null(c))| \leq n-1$
\item\label{vanishing} $\Delta(c) \geq 0$ and $\Delta(c) = 0$ if and only if $\ind(c^m) = 0$ for all $m \in \Z_{>0}.$
\end{enumerate}

Let $\La_{\al}M$ be the $L^2_1$-completion of $\cl L_{\al} M$ and $\La_{\al}^A M = \{\cA < A\}$ for $A>0.$ Note that the inclusion $\cl L^A_{\al} M \to \La^A_{\al} M$ is a homotopy equivalence for every $A>0$ by a result of Anosov \cite{Anosov} (see also \cite[Proposition 2.2]{GH09}).

We require the following property of the homotopy groups of $\La_{\al} M,$ which follows from \cite{GromollMeyer} and the properties of the mean-index (see also \cite{Chang}). 

\begin{prop}\label{prop: Morse}
If all geodesics $b \in \La_{\al} M$ with $\cA(b) \geq A$ are of mean-index $\Delta(b)>k+n-1,$ then $\pi_m(\La_{\al}M) = \pi_m(\La_{\al}^AM)$ for all $m \leq k.$
\end{prop}

We recall the following long exact sequence from \cite{Tai85}. Fix $\al \in \pi_0(\cl L M).$ Consider the evaluation map $ev: \La_{\al} M \to M,$ $z \mapsto z(0),$ where $z \in \La_{\al} M$ is considered as an absolutely continuous map $z:\R/\Z \to M$ with square-integrable velocity vector. This map is a Serre fibration. Fix base-points $x_0 \in M$ and $\gamma_0 \in ev^{-1}(x_0).$ Let $a \in \pi_1(M,x_0)$ be the class represented by $\gamma_0.$ Of course $[a] = \al$ in $\pi_1(M)/\mrm{conj}.$ 

\begin{prop}\label{prop: les}
The long exact sequence of a fibration for $ev$ and an identification $\pi_k(ev^{-1}(x_0),\gamma_0) \cong \pi_{k+1}(M,x_0)$ yield the long exact sequence: {\[ \ldots \to \pi_{m+1}(M,x_0) \to \pi_m(\La_{\al} M, \gamma_0) \to  \pi_{m}(M,x_0) \xrightarrow{1-a} \pi_{m}(M,x_0) \to \ldots,\]}for $m \geq 2,$ where $a: \pi_k(M,x_0) \to \pi_k(M,x_0)$ is the standard action of the fundamental group on higher homotopy groups. For $m=1$ the sequence reads: {\[ \ldots  \to \pi_{2}(M,x_0) \to \pi_1(\La_{\al} M, \gamma_0) \to  C_a \to 1,\] where $C_a \subset \pi_1(M,x_0)$ denotes the centralizer of $a.$}
\end{prop}

{The proof of Proposition \ref{prop: les} is essentially outlined in \cite{Tai85}. For completeness, we give its detailed proof here}.

{ \begin{proof} We have a Serre fibration \[(ev^{-1}(x_0), \gamma_0) \xhookrightarrow{i} (\Lambda_{\alpha}M, \gamma_0) \xrightarrow{ev} (M, x_0),\] where \[i: (ev^{-1}(x_0), \gamma_0) \to (\Lambda_{\alpha}M, \gamma_0)\] is the inclusion. It yields the following long exact sequence of homotopy groups:
\[ \ldots \to  \pi_{m}(ev^{-1}(x_0), \gamma_0) \xrightarrow{i_*} \pi_m(\La_{\al} M, \gamma_0) \xrightarrow{ev_*}  \pi_{m}(M,x_0) \xrightarrow{\delta} \pi_{m-1}(ev^{-1}(x_0), \gamma_0) \to \ldots\]
where $\delta: \pi_{m}(M,x_0) \xrightarrow{\delta} \pi_{m-1}(ev^{-1}(x_0), \gamma_0)$ denotes the connecting morphism. Note that the fiber $ev^{-1}(x_0)$ is in general disconnected: it is the disjoint union $\bigsqcup_{b} \Om_b M$ of the components $\Om_b M$ of the based loop space of $(M,x_0)$, over all $b \in \pi_1(M,x_0)$ with $[b] = \al$ in $\pi_1(M)/\mrm{conj}$; the basepoint $\gamma_0$ lies in the component $\Om_a M$ (so $ev^{-1}(x_0)$ can be also denoted by $\Omega_{\alpha}M$, the part of the based loop space with respect to the free homotopy class $\alpha$, and only the component of $\gamma_0$ enters the based homotopy groups below). 

Moreover, we have a natural identification by an isomorphism of groups,
\[ {\rm eh}: \pi_m(ev^{-1}(x_0),\gamma_0) \xrightarrow{\cong} \pi_{m+1}(M,x_0).\]
Explicitly, concatenation with the reversed loop $\bar\gamma_0$, followed by the standard contraction of $\gamma_0 * \bar\gamma_0$ to the constant loop, defines a homotopy equivalence from $\Om_a M$ to the component $\Om_e M$ of contractible based loops, and on the latter the identification $\pi_m(\Om_e M) \cong \pi_{m+1}(M,x_0)$ holds due to the adjunction relation between the based loop space functor $\Omega$ and the reduced suspension functor $\Sigma$.  Then under this identification, the long exact sequence above is rewritten as follows:
\[ \ldots \to  \pi_{m+1}(M, x_0) \xrightarrow{i_* \circ {\rm eh}^{-1}} \pi_m(\La_{\al} M, \gamma_0) \xrightarrow{ev_*}  \pi_{m}(M,x_0) \xrightarrow{{\rm eh} \circ \delta} \pi_{m}(M, x_0) \to \ldots\]
By Theorem 2.6 in \cite{Lang73}, the new connecting morphism ${\rm eh} \circ \delta$ is equal to the Whitehead product $$[-,[\gamma_0]] = [-, a]: \pi_1(M, x_0) \times \pi_m(M, x_0) \to \pi_m(M, x_0),$$ which by definition, coincides with the standard action of the class $a = [\gamma_0] \in \pi_1(M, x_0)$ on $\pi_m(M, x_0)$. For more details, see Section 3 in \cite{Whi41}. In particular, for $m \geq 2$ and $x \in \pi_m(M, x_0)$, we have $[x, a] =x - a\cdot x = (1-a) \cdot x$. This proves the desired conclusion for $m \geq 2$. 

Finally, for $m =1$ and $x \in \pi_1(M, x_0)$, we have $[x,a] = xax^{-1}a^{-1}$, the commutator in $\pi_1(M, x_0)$ (which can be seen directly without referring to the Whitehead product). Then the long exact sequence above is 
\begin{equation} \label{m1}
\ldots \to  \pi_2(M, x_0) \rightarrow \pi_1(\La_{\al} M, \gamma_0) \xrightarrow{ev_*}  \pi_{1}(M,x_0) \xrightarrow{[-,a]} \pi_{1}(M, x_0) \to 1 
\end{equation}
where the last term $1$ comes from $1 = \pi_0(\La_{\al} M, \gamma_0)$, since $\La_{\al}M$ is by definition a connected component of $\La M$. Then truncated from the last $\pi_1(M, x_0)$ and by the exactness at the second last $\pi_1(M, x_0)$ in (\ref{m1}), we have
\[ \ldots \to  \pi_2(M, x_0) \rightarrow \pi_1(\La_{\al} M, \gamma_0) \xrightarrow{ev_*} {\rm im}(ev_*) = \ker([-,a]) = C_a  \to 1 \]
where $C_a$ denotes the centralizer of $a$. Here, $\ker([-,a]) = \{x \in \pi_1(M,x_0) \mid xax^{-1}a^{-1} = 1\} = C_a$ by the commutator formula above. This proves the desired conclusion for $m = 1$.  \end{proof}
}

{To end this section, let us explain how the hypotheses of Theorem \ref{thm: inv coinv} are used via Proposition \ref{prop: les}. In the invariant case, if $0 \neq x \in \pi_m(M)$ is fixed by every $a \in \pi_1(M),$ then $x \in \ker(1-a),$ so by exactness there is $y_{\al} \in \pi_m(\La_{\al}M)$ with $ev_*(y_{\al}) = x$; in particular $\pi_m(\La_{\al}M) \neq 0$ for every $\al.$ In the coinvariant case, if $0 \neq \xi: \pi_m(M) \to A$ satisfies $a^*\xi = \xi,$ then $\xi \circ (1-a) = 0,$ so $\xi$ descends to a non-zero map on $\coker(1-a),$ which by exactness injects into $\pi_{m-1}(\La_{\al}M)$; in particular $\pi_{m-1}(\La_{\al}M) \neq 0,$ and for $m=2$ the image of $\coker(1-a)$ is a non-trivial subgroup of $\ker(ev_*: \pi_1(\La_{\al}M,\gamma_0) \to C_a).$}

\section{Proof of main result} \label{sec-proof}

\begin{proof}[Proof of Theorem \ref{thm: inv coinv}]
By Proposition \ref{prop: les}, the hypothesis of the theorem implies that $\pi_{\ell}(\La_{\alpha} M) \neq 0$ for $\ell=m$ or $\ell=m-1$ and all $\alpha \in \pi_0(\cl LM).$ Suppose by contradiction that there are finitely many prime closed geodesics. (In particular, every closed geodesic is isolated.) We denote them by $\gamma_1,...,\gamma_P, \gamma_{P+1},...,\gamma_N,$ such that the geodesics $\gamma_1,...,\gamma_P$ have mean-index $\Delta(\gamma_i) > 0$ and $\gamma_i$ for $i>P$ have $\Delta(\gamma_i) = 0.$  Then by Property \eqref{vanishing} of the mean-index, for every $i>P,$ $\ind(\gamma_i^j) = 0$ for all iterations $j.$  

Now all closed geodesics are iterations of $\gamma_1, \ldots, \gamma_N$. Consider the following set
\begin{equation} \label{dist-set}
S_{\ell}:= \big\{ [\gamma_i^j] \;\big|\; 1 \leq i \leq P,\; \Delta(\gamma_i^j) \leq n-1+\ell \big\}\,\,\,\,\mbox{where $n = \dim M$}. 
\end{equation}
Since the mean-indices of $\gamma_i$ for $1 \leq i \leq P$ are positive, by Property \eqref{homogeneity} of the mean-index, $S$ is a finite set. Since $\pi_0(\cl LM)$ is an infinite set, so is the complement $\pi_0(\cl LM) \setminus S.$ Consider $\alpha \in \pi_0(\cl LM) \setminus S.$ Now let $c_1, \ldots, c_r \in \La_{\alpha} M$ be all the points of absolute minimum of the energy functional on $\La_{\al} M$. Then the level set of this minimum is the disjoint union $ \sqcup_{j=1}^r S(c_j)),$ where $S(c_j)$ is the circle given by $c_j$ and its rotations, and therefore its $k$-th homotopy groups vanish for all $k>1.$ We now claim that there exists a geodesic $\eta \in \La_{\alpha} M$ that is not a global minimum and is not an iteration of any of $\gamma_1,...,\gamma_P.$ Note that as $\eta$ must be an iteration of some $\gamma_i$ for $i>P$ we must have $\ind(\eta^r) = 0$ for all $r \geq 1.$ {Moreover, since the metric is bumpy, the critical orbit $S(\eta)$ is non-degenerate; as $\ind(\eta)=0,$ the negative bundle of $S(\eta)$ has rank zero, and is hence orientable. The Morse--Bott lemma then implies that the integral local homology $\rC_k(\eta;\Z) \cong H_k(S^1;\Z)$ is non-zero. So, $\eta$ is homologically visible.} Then by Bangert-Klingenberg \cite[Theorem 3]{BK83}, this implies that there exist infinitely many prime closed geodesics on $M.$

To prove the last claim, we treat the cases $\ell \geq 2$ and $\ell = 1$ slightly differently. Suppose first that $\ell \geq 2.$ The contrapositive assumption that all closed geodesics above the minimum energy are iterations of $\gamma_1, \ldots, \gamma_P$, implies that  $\pi_\ell(\La_{\alpha} M) = 0$ by Proposition \ref{prop: Morse}, since by construction, any iteration $\eta$ of $\gamma_1, \ldots, \gamma_P$ in $\alpha$ has $\Delta(\eta)> n-1+\ell.$ But we know that $\pi_{\ell}(\La_{\alpha} M) \neq 0.$ In the case where $\ell =1,$ we observe that the coinvariant hypothesis and Proposition \ref{prop: les} yields that $\ker(ev: \pi_1(\La_{\alpha} M, \gamma_0) \to \pi_1(M,x_0)) \neq 0.$ Now, we argue that there can only be one point of global minimum under the contrapositive assumption. Indeed, had we $N>1$ global minima $c_1,\ldots,c_N$, the contrapositive assumption, { {together with Proposition \ref{prop: Morse} (applied to $m = 0$),}} would yield that $\pi_0(\La_{\alpha} M)$ is a set of $N$ elements, which contradicts the connectedness of $\La_{\alpha} M.$ Hence there is a unique global minimum $c.$ Now the contrapositive assumption, { {together with Proposition \ref{prop: Morse} again (applied to $m = 1$),}} yields that $\pi_1(\La_{\alpha} M) \cong \pi_1(S(c)) \cong \Z$ {generated by the class of the circle orbit $S(c).$ 

Write $c = \beta^r$ where $\beta$ a prime closed geodesic and $r \geq 1,$ and set $b = [\beta] \in \pi_1(M,x_0)$ (after conjugating so that base-points match). Then $[c] =  b^r.$ Assume $\beta$ has period $1$, then the rotation family $s \mapsto c(\cdot +s)$ is a loop in $\Lambda_{\al} M$ but with period $1/r$. Indeed, 
\[ c\left(t + \frac{1}{r} \right) = \beta(rt + 1) = \beta(rt) = c(t) \]
for all $t \in \R/\Z$. Therefore, the orbit circle $S(c)$ is parametrized by $s \in [0,1/r]$. Meanwhile, by definition, the evaluation map sends a loop of loops to the trace of its base points. Applied to the positively oriented generator of $\pi_1(S(c))$, it maps to
\[ \left[\left\{s \mapsto c(0+s)\right\}_{s \in [0, \frac{1}{r}]}\right] = \left[\{\beta(rs)\}_{s \in [0, \frac{1}{r}]}\right] = b. \]  
Note that it is not the class $[c]$ (where $[c] = b^r$). Then we will proceed in two different cases. If $b$ has infinite order in $\pi_1(M,x_0),$ then} the evaluation map $\pi_1(S(c)) \to C_{a}$ is injective. This is a contradiction to its kernel being non-zero. 

{If instead $b$ has finite order $q \geq 2,$ then $[c^q] = b^{rq} = 1$ in $\pi_1(M),$ and $\Delta(c^q) = q\Delta(c) = 0$ by Property \eqref{homogeneity}. Here, $\Delta(c)=0$ is due to the following reason: since $c$ is a global minimum, $\ind(c) =0$ and then Property \eqref{distance to ind} implies that $|\Delta(c)|\leq n-1$; our choice $\al \notin S$ implies that $c$ cannot be an iteration of any $\gamma_i$ for $1\leq i \leq P$ and in particular, $c = \gamma_i^j$ for some $i>P$, which means $\Delta(c) = j \Delta(\gamma_i) = 0$. Thus the iterate $c^q$ is a non-constant contractible closed geodesic with zero mean-index, and it is not a global minimum of the energy in its (contractible) component (since the global minima there are the constant loops of zero energy). Therefore iterations of $c^q$ satisfy the hypotheses of \cite[Theorem 3]{BK83}, which yields infinitely many geometrically distinct closed geodesics. This, in fact, directly contradicts our contrapositive assumption at the beginning of our proof.} Hence, this proves the claim for all $\ell \geq 1.$  \end{proof}

{
\begin{proof}[Proof of Theorem \ref{thm: every metric}]
Suppose, by contradiction, that $(M,g)$ has finitely many prime closed geodesics, then every closed geodesic is isolated, and, for every non-trivial $\al \in \pi_0(\mathcal L M)$ the set $\rC_{\al}(g)$ of absolute minima in $\La_{\al}M$ is a finite disjoint union of critical circles. Denote by $\gamma_1,\ldots,\gamma_P$ the prime closed geodesics of positive mean-index, and consider the following set (cf.~(\ref{dist-set})) 
\[ S_L := \big\{ [\gamma_i^j] \;\big|\; 1 \leq i \leq P,\; \Delta(\gamma_i^j) \leq n-1+L \big\}. \]
By Property \eqref{homogeneity}, it is a finite set. By our assumption, there are infinitely many classes $\al$ satisfy \eqref{eq: rel homology hyp}, so we may choose a non-trivial class $\al \notin S_L.$ Then every closed geodesic $c$ in $\Lambda_{\al} M$ with $\Delta(c) > 0$ satisfies $\Delta(c) > n-1+L$. By Property \eqref{distance to ind}, we have 
\[ \ind(c) \geq \Delta(c) - (n-1) > L \geq j. \]
By a celebrated result from Gromoll--Meyer, say Lemma 1 with Theorem in Section 3 in \cite{GromollMeyer-degenerate} (and also \cite{GromollMeyer}), the integral local homology of $S(c)$ vanishes in degree $j$ (or directly from Corollary A.5.4 in \cite{Maz12}). Note that this works even when no non-degeneracy condition is assumed.  
	
Next, we focus on those closed geodesics $c$ in $\Lambda_{\al}M$ with $\Delta(c)=0$. We claim that there exists some closed geodesic $\eta \in \Lambda_{\al} M$ with $\Delta(\eta) = 0,$ not an absolute minimum, but has non-vanishing integral local homology. Then \cite[Theorem 3]{BK83} yields the desired contradiction. Suppose not, that is, {\it every} closed geodesic $\eta \in \Lambda_{\al} M$ with $\Delta(\eta) = 0$ that is not an absolute minimum satisfies $\rC_*(\eta;\Z) =0$ for every degree $*$, in terms of the notation in (\ref{dfn-int-loc-hom}). In particular, $\rC_j(\eta;\Z)=0$. 

Since the energy functional $\mathcal A$ on $\Lambda_{\al}M$ satisfies the Palais--Smale condition, only finitely many critical orbits below any fixed level. So, we may filter $\Lambda_{\al}M$ by sublevel sets crossing one critical level at a time. For instance, choose $A \leq B$ such that $\Lambda^A_{\al}M \subset \Lambda^B_{\al}M$ are sublevel sets with exactly one critical level $\kappa$ in between, then 
\begin{equation} \label{rel-vanishing}
H_j(\Lambda^B_{\al}M, \Lambda^A_{\al}M; \Z) = \bigoplus_{\mathcal A(\eta)= \kappa} \rC_j(\eta; \Z) = 0.
\end{equation}
Moreover, consider the exact sequence of the triple $(\Lambda^B_{\al}M,  \Lambda^A_{\al}M, \rC_g(\alpha))$, where recall that $\rC_g(\alpha)$ denote the set of absolute minimum of the energy functional $\mathcal A$ on $\La_{\al}M$: 
 \begin{equation} \label{rel-esc}
  \cdots \to  H_j(\Lambda^A_{\al}M, \rC_g(\alpha); \Z) \to H_j(\Lambda^B_{\al}M, \rC_g(\alpha);\Z) \to H_j(\Lambda^B_{\al}M, \Lambda^A_{\al}M; \Z) \to \cdots. 
  \end{equation} 
 For the initial sublevel $\Lambda^A_{\al}M$ where $A$ is just above the minimal level, $\Lambda^A_{\al}M$ deformation retracts onto $\rC_g(\alpha)$ (as the absolute minima are finitely many isolated critical circles), so $H_j(\Lambda^A_{\al}M, \rC_g(\alpha);\Z) = 0.$ Then (\ref{rel-vanishing}) implies that at the next level (for some $B$), we have $H_j(\Lambda^B_{\al}M, \rC_g(\alpha);\Z) = 0$. Inductively applying (\ref{rel-vanishing}) and (\ref{rel-esc}) over the (discrete) critical levels and passing to the limit, we get $H_j(\Lambda_{\al}M, \rC_g(\alpha);\Z) = 0$. This contradicts our assumption (\ref{eq: rel homology hyp}), and thus completes the proof. 
\end{proof}
}

\section*{AI use declaration}
The application of the Tits alternative in Example \ref{exa:new} was suggested by ChatGPT 5.6. We take full responsibility for the mathematical content and conclusions of this example.

\section*{Acknowledgments}
We thank Marco Mazzucchelli, as well as the participants of the reading group on closed geodesics at the Universit\'e de Montr\'eal, for useful discussions. {We are grateful to anonymous referees for pointing out the homological visibility hypothesis in \cite[Theorem 3]{BK83} missing in an early version of this paper, and for prompting us to include a complete proof of Proposition \ref{prop: les}.}
E.S. was supported by an NSERC Discovery grant, by the Fonds de recherche du Qu\'{e}bec - Nature et technologies, by the Fondation Courtois, by an Alfred P. Sloan Research Fellowship, and by an anonymous donor as a member in the Institute for Advanced Study during the Fall semester of 2026. This work was partially supported by the National Science Foundation under Grant No. DMS-1928930, while E.S. was in residence at the Simons Laufer Mathematical Sciences Institute (previously known as MSRI) Berkeley, California during the Fall semester of 2022, and by the National Science Foundation under Grant No. DMS-1926686 while E.S. was in residence at Institute for Advanced Study during the academic year 2024-2025. 
J.Z. was supported by the USTC Research Funds of the Double First-Class Initiative during the writing of the first version of this paper, and is currently supported by the National Key R\&D Program of China (Grant No.~2023YFA1010500) and NSFC Grants No.~12301081 and No.~12511540054. This work was initiated when J.Z. was a CRM-ISM Postdoctoral Research Fellow at CRM, Universit\'e de Montr\'eal, and he thanks this institute for its warm hospitality.

\bibliography{RFG_cited}

\medskip 
\end{document}